\documentclass{article}
\usepackage[english]{babel}
\usepackage{amsthm}
\usepackage{amssymb}
\usepackage{amsmath}
\usepackage{amsfonts}
\usepackage{color}
\theoremstyle{definition}
\newtheorem{theorem}{Theorem}[section]
\newtheorem{definition}{Definition}[section]
\newtheorem{example}{Example}[section]
\newtheorem{remark}{Remark}[section]
\title{Logical connectives of fuzzy soft set theory}
\date{}
\begin{document}
\author{Santanu Acharjee$^1$ and Sidhartha Medhi$^2$\\
$^{1,2}$Department of Mathematics\\
Gauhati University\\
Guwahati-781014, Assam, India.\\
e-mails: $^1$sacharjee326@gmail.com,
$^2$sidharthamedhi365@gmail.com\\
Orchid: $^1$0000-0003-4932-3305,
$^2$0009-0001-6692-3647}

\maketitle
{\bf Abstract:} 
Soft set theory, introduced by Molodtsov [Molodtsov, D. (1999). Soft set theory—first results. Comput. Math. Appl., 37(4-5), 19-31], provides a flexible framework for managing uncertainty and vagueness, addressing limitations in traditional approaches such as fuzzy set theory, rough set theory, and probability theory. Over time, fuzzy soft set theory has emerged as a significant extension, blending the principles of fuzzy set theory and soft set theory to support applications in various decision-making processes. This study revisits fuzzy soft set theory, addressing conceptual errors and inaccuracies in the definitions of t-norm, t-conorm, strong negation, and implication that deviated from Molodtsov's foundational principles. Corrected definitions—fuzzy soft t-norm, fuzzy soft t-conorm, fuzzy soft  negation, and  fuzzy soft implication—are proposed to ensure theoretical rigor. The paper rectifies conceptual errors in prior work by Ali and Shabir [Ali, M. I.,  Shabir, M. (2013). Logic connectives for soft sets and fuzzy soft sets. IEEE Transactions on Fuzzy Systems, 22(6), 1431-1442.] and introduces refined results to strengthen the logical framework, providing a consistent foundation for future research and hybrid model development in this domain.\\


{\bf 2020 AMS classifications: } 03B52; 03E72;  94D05.\\

\section{Introduction}
Soft set theory, introduced by Molodtsov \cite{1} in 1999, is designed to deal with uncertainties and vagueness. Conventional methods, such as interval mathematics, fuzzy set theory, rough set theory, and probability theory, face challenges in managing uncertainties due to their dependence on specific parameterizations, membership values, etc. \cite{1}. In comparison, soft set theory provides a more flexible and generalized approach, linking attributes with their approximations without requiring predefined membership values. Since Molodtsov's introduction of soft set theory, many researchers have expanded its concepts into various generalized forms, exploring its operations, relations, and applications in diverse fields. Maji et al. \cite{7} introduced the concept of fuzzy soft theory and established various fundamental notions, such as complement, union, intersection, etc., of fuzzy soft sets to examine several key properties.\\

Moreover, many researchers have adopted various notions proposed by Maji et al. \cite{7} and developed many new concepts in the area of fuzzy soft set theory. However, compared to Molodtsov's studies \cite{2,3,4,10}, their results are found to be incorrect. In \cite{3}, Molodtsov defined the correct concepts related to various operations in soft set theory. In \cite{2, 3}, it is clearly stated that many researchers have introduced new operations and related results for soft sets without considering the specifics of the definition of a soft set and its related concepts. Molodtsov \cite{2,3} mentioned, with proper justifications, that Maji et al. \cite{11} introduced incorrect definitions of the empty set, complement, union, intersection, and subset in soft set theory. Yet, almost all researchers have been following these incorrect notions, thus making soft set theory and related hybrid structures erroneous in every aspect. In \cite{7}, Maji et al. defined the soft subset of a soft set as follows:\\

``For two soft sets $(F, A)$ and $(G, B)$ over a common universe set $U$, we say that $(F, A)$ is a soft subset of $(G, B)$ if $A\subset B$ , and $\forall \epsilon \in A$, $F(\epsilon)$ and $G(\epsilon)$ are identical approximations." \\

But Molodtsov \cite{3,4} mentioned that the condition $A\subset B$ or $A\subseteq B$ leads to a lack of conceptual correctness, since he did not conceptualize a soft subset anywhere in the literature of soft set theory \cite{1,2,3} or in his book \cite{12}. According to him \cite{2,3,12}, correct operations should be defined through the families of sets $\tau(S,A)$. One can refer to Acharjee and Oza \cite{4}, Acharjee and Molodtsov \cite{9}, and Acharjee and Medhi \cite{10} for more information related to the above facts. Molodtsov \cite{3} also stated that ``...the generation of hybrid approaches combining soft sets and fuzzy sets should be carried out very carefully, taking into account differences between the concepts of equivalence." Therefore, it is crucial to recognize and resolve these discrepancies to ensure that the resulting hybrid structures are both theoretically robust and practically effective.\\

 Ali and Shabir \cite{6} proposed logical connectives, such as t-norms, t-conorms, strong negation, and implication, in fuzzy soft set theory. These definitions have been found to deviate from the core principles of Molodtsov's original theory of soft sets, despite their goal of improving the mathematical framework of fuzzy soft sets and related logical connectives. The variation required a careful re-examination and revision of these notions in fuzzy soft set theory to preserve consistency within the theoretical framework of soft set theory. It is important to note that Molodtsov's definitions of the null soft set \cite{2} and the absolute soft set \cite{3} are conceptually different from the available definitions of the null soft set \cite{11} and the absolute soft set \cite{11}. Hence, there are also similar conceptual errors in the definitions of the null fuzzy soft set \cite{11} and the absolute fuzzy soft set \cite{11}. All these conceptual errors are present in the literature of fuzzy soft set theory, except in Acharjee and Medhi \cite{10}. The reflections of these errors can also be found in definitions 25, 27, 28, 30, and related results of Ali and Shabir \cite{6}. We discuss these in the next section.\\

To resolve the aforementioned disparities, this study provides correct definitions of logical connectives such as t-norm, t-conorm, strong negation, and implication in fuzzy soft set theory. We refer to these corrected definitions as fuzzy soft t-norm, fuzzy soft t-conorm, fuzzy soft negation, and fuzzy soft implication in fuzzy soft set theory.\\

These definitions and related results have been thoughtfully developed in accordance with the fundamental principles of Molodtsov's soft set theory \cite{2,3,10,12}. The contributions of this work are twofold. First, it rectifies the inaccuracies in the previously proposed definitions and results of Ali and Shabir \cite{6} related to t-norm, t-conorm, strong negation, and implication in fuzzy soft set theory, ensuring their theoretical soundness and practical relevance. Second, it advances the logical development of fuzzy soft sets, providing new insights and tools for future research in this domain. By doing so, this paper aims to serve as a foundational piece of research for scholars and practitioners seeking correct and consistent methodologies within the logical connectives of fuzzy soft set theory.\\

The paper is divided into four sections. Section 1 is the introduction. Preliminary definitions are provided in Section 2. Section 3 presents the main results of this research, and Section 4 concludes the paper.\\

\section{Preliminaries}
A soft set is a mathematical structure consisting of two essential parts: a universal set $X$ and a set of parameters $A$. The basic concepts and related results of soft set theory, fuzzy soft set theory, and their associated theories are cited in this section. These will be used in the next section.\\

\begin{definition}\cite{1, 2,3}
A soft set defined over a universal set $X$ is a pair $(S,A)$ where $S$ is a mapping from the set of parameters $A$ into the set of all subsets of $X$, i.e., $S : A \to 2^X$. In fact, a soft set is a parametrized family of subsets. For a given soft set $(S,A)$, the family $\tau(S,A)$ can be defined as $\tau(S,A)=\{S(a): a \in A\}$. \\
    
\end{definition}

\begin{definition}\cite{10}
    A fuzzy soft set is a pair $(S,A)$ defined over a universal set $X$ if and only if $S$ is a mapping from the set of parameters $A$ into the set of all fuzzy subsets of $X$, i.e., $S : A\to P(X)$. So, in a fuzzy soft set $(S, A)$, each $S(a)$, $\forall a \in A$, is a fuzzy set. For a given fuzzy soft set (S,A), the family $\tau (S,A)$ can be defined as the set $\tau (S,A)=\{S(a): a\in A \}$.\\

    Here, $P(X)$ denotes the set of all fuzzy sets defined over  $X$.\\
\end{definition}

\begin{definition}\cite{10}
    Let $(S, A)$ be a fuzzy soft set defined over a universal set $X$. Then, the complement of $(S, A)$ is denoted by $C(S, A)$ and defined as $CS(a)= \overline {S(a)}$, $\forall a\in A$.\\
\end{definition}

\begin{definition}\cite{10}
    Let $(S, A)$ and $(G,B)$ be two fuzzy soft sets defined over a universal set $X$. Then, the binary operation union of these fuzzy soft sets is again a fuzzy soft set $(U, A \times B)$ over the same universal set $X$, i.e., $(S,A) \cup (G,B)=(U, A\times B)$ and it is defined by     $U(a, b)=S(a)\cup G(b)$, $\forall (a ,b) \in A\times B$.\\

Here, the membership value of  an element $x$ of the fuzzy set $S(a)\cup G(b)$ is defined as $\mu_{U(a,b)}(x) = \mu_{S(a)\cup G(b)}(x) = max\{\mu_{S(a)}(x), \mu_{G(b)}(x)\}$, $\forall x\in X$.\\
\end{definition}

\begin{definition}\cite{10}
     Let $(S, A)$ and $(G,B)$ be two fuzzy soft sets defined over a universal set $X$. Then, the binary operation intersection of these fuzzy soft sets is again a fuzzy soft set $(I, A \times B)$ over the same universal set X, i.e., $(S,A) \cap (G,B)=(I, A\times B)$ and it is defined by $I(a, b)=S(a)\cap G(b)$ , $\forall (a ,b) \in A\times B$.\\

Here, the membership value of an element $x$ of the fuzzy set $S(a)\cap G(b)$ is defined as $\mu_{I(a,b)}(x) = \mu_{S(a)\cap G(b)}(x) = min\{\mu_{S(a)}(x), \mu_{G(b)}(x)\}$, $\forall x\in X$.\\
\end{definition}

\begin{definition}\cite{13}
    A fuzzy implication $R$ is a binary operation from $[0, 1]\times [0, 1] \to [0, 1]$ such that the following properties hold:

    (i) $R(p, r)\geq R(q, r)$ if $q\geq p$;
    
    (ii) $R(p, r)\leq R(p, s)$ if $r\leq s$;

    (iii) $R(1, t)=t$, $\forall t\in [0, 1]$;

    (iv) $R(0, t)=1$, $\forall t\in [0, 1]$;

    (v) $R(p, R(q, r))=R(q, R(p, r))$.\\
\end{definition}

Ali and Shabir \cite{6} introduced logical connectives in soft set theory and fuzzy soft set theory. According to their definitions \cite{6}, the t-norm, t-conorm, strong negation, and implication in fuzzy set theory are as follows:\\

\begin{definition}\cite{6}
    A binary operation $T$ on $(FS)_E$ is called a t-norm on $(FS)_E$, if the following  hold:\vspace{1.5mm}
    
    i) $T(\mathcal{U}_E, (F, A))= (F, A)$,

    ii) $T((F, A), (G, B))=T((G, B), (F, A))$,

    iii) $T((F, A), T((G, B), (H, C))) = T(T((F, A), (G, B)), (H, C))$,

    iv) if $(F, A)\tilde{\subset} (G, B)$ and  $(H, C)\tilde{\subset} (K, D)$, then $T((F, A), (H, C)) \tilde{\subset} T((G, B), (K,\\ D))$.\\

    \vspace{1.5mm}

Here, $(FS)_E$ is the collection of all fuzzy soft sets defined over a universal set $U$, $\mathcal{U}_E$ is an absolute fuzzy soft set  and  ``$(F, A)$ is a fuzzy soft subset of $(G, B)$"  is represented as $(F, A)\tilde{\subset} (G, B)$ in the sense of  Maji et al. \cite{7}.\vspace{1.5mm}
\end{definition}

\begin{definition}\cite{6}
    Let $\mathcal{N}:(FS)_E \to (FS)_E$ be a map. It is called a strong negation if the following hold:

    (i) $\mathcal{N}(U_A)={\emptyset}_A$ and $\mathcal{N}({\emptyset}_A)=U_A$,

    (ii) if $(F, A) \tilde{\subset} (G, A)$ then $\mathcal{N}(G, A)\tilde{\subset} \mathcal{N}(F, A)$,

    (iii) $\mathcal{N}(\mathcal{N}(F, A))=(F, A)$.
\end{definition}

\begin{definition}\cite{6}
    A binary operation $T$ on $(FS)_E$ is called the dual of another binary operation $S$ on $(FS)_E$ with respect to the strong negation $\mathcal{N}$, if any of the following hold:\\
   \begin{center}
       $S((F, A), (G, B))=T(\mathcal{N}(\mathcal{N}(F, A), \mathcal{N}(G, B)))$,
   \end{center} 
    
    and\\

    \begin{center}
         $T((F, A), (G, B))=S(\mathcal{N}(\mathcal{N}(F, A), \mathcal{N}(G, B)))$.\\
    \end{center}
   
    If $T$ is a t-norm, then $S$ is called a t-conorm.
\end{definition}

\begin{definition}\cite{6}  A map $S:(FS)_E \times (FS)_E \to (FS)_E$ is a t-conorm on $(FS)_E$, if and only if the following conditions hold:\vspace{1.5mm}
    
    i) $S(\emptyset_E, (F, A))= (F, A)$,

    ii) $S((F, A), (G, B))=S((G, B), (F, A))$,

    iii) $S((F, A), S((G, B), (H, C)))=S(S((F, A), (G, B)), (H, C))$,

    iv) if $(F, A)\tilde{\subset} (G, B)$ and  $(H, C)\tilde{\subset} (K, D)$, then $S((F, A), (H, C)) \tilde{\subset} S((G, B), (K, \\D))$.\vspace{1.5mm}

Here, $\emptyset_E$ is the null fuzzy soft set.
\end{definition}

In these aforementioned definitions of  t-norm and t-conorm, conditions (i) and (iv) are not correct compared to Molodtsov's soft set theory \cite{1,2,3,4,10, 12}. Similarly, conditions (i) and (ii) of definition 2.8 are not correct. We already discussed these in the previous section. Moreover, Molodtsov clearly wrote that, in the case of binary operations, the set of parameters of the resulting soft set for any two soft sets should be the Cartesian product of the sets of parameters of the original soft sets. Thus, the conditions (i) of definitions 2.7 and 2.10, respectively, are not correct since $E\times A\neq A$. Again, as we discussed previously, the notion of a soft subset does not exist in soft set theory \cite{2,3, 12}. Hence, the concept of a fuzzy soft subset also does not exist in fuzzy soft set theory since the only difference, with soft set theory, is that attribute approximations are fuzzy sets and operations are in terms of fuzzy sense in fuzzy soft theory. Thus, the conditions (iv) of definitions 2.7 and 2.10 have no meaning in terms of fuzzy soft theory. Hence, we can say that definitions of t-norm and t-conorm are not correct from the perspective of fuzzy soft set theory and related structures. Thus, we want to propose correct definitions of fuzzy soft t-norm and fuzzy soft t-conorm by following Molodtsov's correct structure of soft set theory \cite{1,2,3,4,10, 12}. They are discussed in the next section.

\section{Main Results}
In this section, we provide correct definitions of fuzzy soft t-norm, fuzzy soft t-conorm, fuzzy soft negation, and fuzzy soft implication based on Molodtsov's concepts \cite{1,2,3,4,10,12}.  However, we urge researchers not to look into soft set theory and its hybrid structures only from the perspective of mathematics, since there are several  philosophical connections behind Molodtsov's ideas on soft set theory. In \cite{2}, Molodtsov clearly mentioned the following:\\

``{\it If the soft set $(S.A)$ is given, then the family $\tau(S,A) = \{ S(a): a\in A\}$ specifies those subsets which can be approximate descriptions, and the parameter set $A$ is chosen for reasons of convenience by the person who introduces the definition of this soft set. Thus, the role of a set of parameters in the definition of a soft set is purely auxiliary. Parameters are used only as short names of subsets.}"\\

From the above comment of Molodtsov, it can be found that logical  connectives of fuzzy set theory deal with membership values, irrespective of the elements that are connected to the membership values. To define logical connectives of fuzzy soft set theory, it is easy to find that a fuzzy set, say $S(a)$, is only the approximation of the attribute, say  $a\in A$. In layman's language, we can tell that $S(a)$ is the name of the fuzzy set. Moreover, in soft set theory, sets of parameters $A\times B$, and $B\times A$ are indistinguishable \cite{2, 3, 4}. Thus, they can be replaced by each other in the theories of soft sets and hybrid structures whenever necessary. Similarly, the sets of parameters $A\times B\times C$, $A\times (B\times C)$ and $(A\times B)\times C$ are indistinguishable in soft set theory and hybrid structures. Thus, considering the aforementioned philosophical aspects of soft set theory and hybrid structures,   we have the following definitions and results on logical connectives of fuzzy soft set theory for any arbitrary sets of parameters $A, B$ and $C$:\\

\begin{definition}
    An operation $T:(A\times[0, 1])\times (B\times[0, 1]) \to (A\times B)\times[0, 1]$ is said to be a fuzzy soft t-norm if it satisfies the following conditions: 

    (i) $T((a, 1), (b, \beta))=((a, b), \beta)$,

    (ii) $T((a, \alpha), (b, 1))=((a, b), \alpha)$,
    
    (iii) $T((a, \alpha), (b, \beta))=T((b, \beta), (a, \alpha))$,

    (iv) $T((a, \alpha), T((b, \beta), (c, \gamma)))=T(T((a, \alpha), (b, \beta)), (c, \gamma))$,

    (v) if $\alpha_1 \leq \alpha_2$ and $\beta_1\leq \beta_2$, then $T((a, \alpha_1), (b, \beta_1))  \leq T((a, \alpha_2), (b, \beta_2))$.\\

    Here, $A, B, C$ are the sets of parameters, where $a\in A, b\in B,$  and $ c\in C$. Also, $\alpha, \alpha_1, \alpha_2, \beta, \beta_1, \beta_2, \gamma \in [0, 1]$.\\
\end{definition}
\begin{remark}
   Let $(S, A)$ be a fuzzy soft set defined over a universal set $X$. A pair $(a, \alpha)\in A\times [0, 1]$ indicates that $\alpha$ is the membership value of an element of $X$ in the fuzzy set $S(a)\in \tau(S, A)$.
\end{remark}

\begin{definition}
Let $(a, x), (a, y)\in A\times[0, 1]$. Then, we say

(i) $(a, x)\leq (a, y)$ if and only if $x\leq y$,

(ii) $(a, x)\geq (a, y)$ if and only if $x\geq y$.
\end{definition}

 \begin{example}
     Let $T:(A\times[0, 1])\times (B\times[0, 1]) \to (A\times B)\times[0, 1]$ be an operation defined as $T((a, x), (b, y))=((a, b), xy)$, where $A$ and $B$ are sets of parameters. Then, $T$ is a fuzzy soft t-norm.\\

 \begin{proof}
     To show that $T$ is a fuzzy soft t-norm, we first consider three fuzzy soft sets $(S, A)$, $(G, B)$ and $(H, C)$ defined over a universal set $X$. We know that the sets of parameters $A\times B\times C$, $A\times (B\times C)$ and $(A\times B)\times C$ are indistinguishable in fuzzy soft set theory. Let $a\in A, b\in B, c\in C \;$and $\; x, y, z\in [0, 1]$.\\

Let $a\in A, b\in B, \;and\; c\in C$. Also, we consider $x, y, z\in [0, 1]$. Thus, we have the following results:\\

     (i) $T((a, 1), (b, y))=((a, b), 1.y)=((a, b), y)$.\vspace{1.5mm}

     (ii) $T((a, x), (b, 1))=((a, b), x.1)=((a, b), x)$.\vspace{1.5mm}

     (iii) Since $A\times B$ and $B\times A$ are indistinguishable,  we get $T((a, x), (b, y))=((a, b), xy)=((b, a), yx)=T((b, y), (a, x))$.  Hence, $T((a, x), (b, y))=T((b, y), (a, x))$.\vspace{1.5mm}

     (iv) $T((a, x), T((b, y), (c, z)))=T((a, x), ((b, c), yz))=((a, (b, c)), x(yz))$. Since $A\times (B\times C)$ and $(A\times B)\times C$ are indistinguishable, we get $((a, (b, c)), x(yz))=(((a, b), c), (xy)z)=T(((a, b), xy), (c, z))=T(T((a, x), (b, y)), (c, z))$. Hence, $T((a, x), T((b, y), (c, z)))=T(T((a, x), (b, y)), (c, z))$.\vspace{1.5mm}

     (v) Let $x_1\leq x_2$, $y_1\leq y_2$ and $x_1, x_2, y_1, y_2 \in [0,1]$. It implies $x_1y_1\leq x_2y_2$. So, $T((a, x_1), (b, y_1))=((a, b), x_1y_1)\leq ((a, b), x_2y_2)= T((a, x_2), (b, y_2))$.\vspace{1.5mm}

     Thus, $T$ satisfies all the conditions of a fuzzy soft t-norm. Hence, we can say that $T$ is a fuzzy soft t-norm. 
 \end{proof}
 \end{example}
\begin{example}
     The following are examples of fuzzy soft t-norms:\\
    (i) $T((a, x), (b, y))=((a, b), min\{x, y\})$, \\
    (ii) $T((a, x), (b, y))=((a, b), max\{x+y-1, 0\})$.\\
\end{example}

\begin{definition}
     An operation $T':(A\times[0, 1])\times (B\times[0, 1]) \to (A\times B)\times[0, 1]$ is said to be a fuzzy soft t-conorm if it satisfies the following conditions: 

    (i) $T'((a, 0), (b, \beta))=((a, b), \beta)$,

    (ii) $T'((a, \alpha), (b, 0))=((a, b), \alpha)$,
     
    (iii) $T'((a, \alpha), (b, \beta))=T'((b, \beta), (a, \alpha))$,

    (iv) $T'((a, \alpha), T'((b, \beta), (c, \gamma)))=T'(T'((a, \alpha), (b, \beta)), (c, \gamma))$,

    (v) if $\alpha_1 \leq \alpha_2$ and $\beta_1\leq \beta_2$ then $T'((a, \alpha_1), (b, \beta_1))\\ \leq T'((a, \alpha_2), (b, \beta_2))$.

 Here, $A, B, C$ are the sets of parameters, where $a\in A, b\in B,$  and $ c\in C$. Also, $\alpha, \alpha_1, \alpha_2, \beta, \beta_1, \beta_2, \gamma \in [0, 1]$.\\

\end{definition}

\begin{example} The following are examples of fuzzy soft t-conorms:\\ (i) $T'((a, x), (b, y))=((a, b), max\{x, y\})$,\\  
(ii) $T'((a, x), (b, y))=((a, b), x+y-xy)$,\\ 
(iii) $T'((a, x), (b, y))=((a, b), min\{1, x+y\})$ 
\end{example}

\begin{theorem}
    An operation $T:(A\times[0, 1])\times (B\times[0, 1]) \to (A\times B)\times[0, 1]$, defined as $T((a, x), (b, y)) =  ((a, b), t(x, y))$, is fuzzy soft t-norm if $t$ is a fuzzy t-norm.
\end{theorem}
\begin{proof}
     Consider a fuzzy t-norm $t:[0, 1]\times [0, 1] \to [0, 1]$. Then we have the following conditions:
    
    (i) $t(x, 1)=x$ and $t(1, y)=y$, $\forall x, y \in [0, 1]$.
    
    (ii) $t(x, y)=t(y, x)$, $\forall x, y \in [0, 1]$.
    
    (iii) $t(x, t(y, z))=t(t(x, y), z)$, $\forall x, y, z\in [0, 1]$.

    (iv) $x \leq y \implies t(z, x)\leq t(z, y)$, $\forall x, y, z\in [0, 1]$.\\

Now, to prove that $T$ is a fuzzy soft t-norm, we need to show that $T$ satisfies the conditions of definition 3.1. Let $a\in A, b\in B, c\in C \;$ and $\; x, y, z\in [0, 1]$. Thus, \vspace{1.5mm}

(i) $T((a, 1), (b, y))=((a, b), t(1, y))=((a, b), y)$.

(ii) $T((a, x), (b, 1))=((a, b), t(x, 1))=((a, b), x)$.

(iii) $T((a, x), (b, y))=((a, b), t(x, y))=((a, b), t(y, 
x))=((b, a), t(y, x))\\=T((b, y), (a, x))$.

(iv) $T((a, x), T((b, y), (c, z)))=T((a, x),((b, c), t(y, z)))=((a, (b, c)), t(x, t(y, z)))\\=(((a, b), c), t(t(x, y), z))=T(((a, b), t(x, y)), (c, z))=T(T((a, x), (b, y)), (c, z))$.

(v) Let $x_1\leq x_2$ and $y_1\leq y_2$, where $x_1, x_2, y_1, y_2 \in [0, 1]$. Now, $T((a, x_1), (b, y_1))=((a, b), t(x_1, y_1)) \leq ((a, b), t(x_1, y_2)) \leq ((a, b), t(x_2, \\y_2))=T((a, x_2), (b, y_2))$.

Thus, the given operation $T$ satisfies all the conditions of a fuzzy soft t-norm. Hence, $T$ is a fuzzy soft t-norm.
\end{proof}

\begin{theorem}
     An operation $T':(A\times[0, 1])\times (B\times[0, 1]) \to (A\times B)\times[0, 1]$, defined as $T'((a, x), (b, y)) = ((a, b), s(x, y))$, is fuzzy soft t-conorm if $s$ is a fuzzy t-conorm.
\end{theorem}
\begin{proof}
     Let $s:[0, 1]\times [0, 1] \to [0, 1]$ be a fuzzy t-conorm. Thus, we have the following conditions:
    
    (i) $s(x, 0)=x$ and $s(0, y)=y$, $\forall x, y \in [0, 1]$.
    
    (ii) $s(x, y)=s(y, x)$, $\forall x, y \in [0, 1]$.
    
    (iii) $s(x, s(y, z))=s(s(x, y), z)$, $\forall x, y, z\in [0, 1]$.

    (iv) $x \leq y \implies s(z, x)\leq s(z, y)$, $\forall x, y, z\in [0, 1]$.\\
    
Now, for any $a\in A, b\in B, c\in C $; and $ x, y, z\in [0, 1]$, we get,\\

(i) $T'((a, 0), (b, y))=((a, b), s(0, y))=((a, b), y)$.\\

(ii) $T'((a, x), (b, 0))=((a, b), s(x, 0))=((a, b), x)$.\\

(iii) $T'((a, x), (b, y))=((a, b), s(x, y))=((a, b), s(y, 
x))=((b, a), s(y, x))\\=T'((b, y), (a, x))$.\\

(iv) $T'((a, x), T'((b, y), (c, z)))=T'((a, x),((b, c), s(y, z)))=((a, (b, c)), s(x, s(y, z)))\\=(((a, b), c), s(s(x, y), z))=T'(((a, b), s(x, y)), (c, z))=T'(T'((a, x), (b, y)), (c, z))$.\\

(v) Let $x_1\leq x_2$ and $y_1\leq y_2$, where $x_1, x_2, y_1, y_2 \in [0, 1]$. Now, $T'((a, x_1), (b, y_1))=((a, b), s(x_1, y_1)) \leq ((a, b), s(x_1, y_2)) \leq ((a, b), s(x_2, y_2))=T'((a, x_2), (b, y_2))$.\\

Thus,  $T'$ is a fuzzy soft t-norm.\\
\end{proof}
\begin{definition}
    Let $T:(A\times[0, 1])\times (B\times[0, 1]) \to (A\times B)\times[0, 1]$ be a fuzzy soft t-norm defined as $T((a, x), (b, y)) = ((a, b), f(x, y))$, where $f:[0, 1]\times [0, 1] \to [0, 1]$ is a function. Then, the fuzzy soft t-norm $T$ is said to be continuous fuzzy soft t-norm if the function $f$ is continuous. 
\end{definition}
\begin{example}
    The following are examples of continuous fuzzy soft t-norms:\\
    
(i) $T((a, x), (b, y)) = ((a, b), xy)$,

    (ii) $T((a, x), (b, y)) = ((a, b), min\{x, y\})$.
\end{example}
\begin{definition}
    Let $T':(A\times[0, 1])\times (B\times[0, 1]) \to (A\times B)\times[0, 1]$ be a fuzzy soft t-conorm defined as  $T'((a, x), (b, y)) = ((a, b), g(x, y))$, where $g:[0, 1]\times [0, 1] \to [0, 1]$ is a function. Then, the fuzzy soft t-conorm $T'$ is said to be continuous fuzzy soft t-conorm if the function $g$ is continuous. 
\end{definition}
\begin{example}
     Let $T':(A\times[0, 1])\times (B\times[0, 1]) \to (A\times B)\times[0, 1]$ be a fuzzy soft t-conorm defined as $T'((a, x), (b, y)) = ((a, b), max\{x, y\})$. Then, $T`$ is a continuous fuzzy soft t-conorm.
\end{example}
\begin{definition}
    An operation $N:A\times[0, 1] \to A\times[0, 1]$, where $A$ be a set of parameters, is said to be a fuzzy soft negation if it satisfies the following conditions:

    (i) $N(a, 1)=(a, 0)$ and $N(a, 0)=(a, 1)$, $\forall a\in A$.

    (ii) if $x\leq y$, then $N(a, x)\geq N(a, y)$, $\forall a\in A$ and $x, y \in [0, 1]$.

    (iii) $N(N(a, x))=(a, x)$, $\forall a\in A$ and $x \in [0, 1]$.
\end{definition}

\begin{example}
    Let $N:A\times[0, 1] \to A\times[0, 1]$ be an operation defined as $N(a, x)=(a, 1-x)$, $\forall a\in A$ and $x\in [0, 1]$. Then, $N$ is a fuzzy soft negation.
\end{example}

\begin{definition}
    Let $N:A\times[0, 1] \to A\times[0, 1]$ be a fuzzy soft negation. Then, $(a, x)$ is an equilibrium point of $N$ if $N(a, x)=(a, x)$.
\end{definition}

Now, we prove the following theorems:\\

\begin{theorem}
    Every fuzzy soft negation operation $N:A\times[0, 1] \to A\times[0, 1]$ has at most $n$ equilibrium points, where $n$ be the cardinality of the set of parameters $A$.
\end{theorem}
\begin{proof}
    Let $N:A\times[0, 1] \to A\times[0, 1]$ be a fuzzy soft negation and let $A=\{a_1, a_2, \dots, a_n\}$ be the set of parameters. We know that if $(a, x)\in A\times [0, 1]$ be an equilibrium point of $N$, then we have $N(a, x)=(a, x)$ or $N(a, x)-(a, x)=0$.

    We first prove that $N$ has at most one equilibrium point for each $a_i \in A$, where $i=1, 2, \dots, n$. 
    
    For the above purpose, we assume that $N$ has two different equilibrium points with respect to the parameter $a_1$, say $(a_1, x), (a_1, y) \in A\times [0, 1]$. Therefore, $(a_1, x)$ and $ (a_1, y)$ are the solutions of the equation $N(a, x)-(a, x)=0$, i.e., $N(a_1, x)-(a_1, x)=0$ and $N(a_1, y)-(a_1, y)=0$. This implies that $N(a_1, x)-(a_1, x)=N(a_1, y)-(a_1, y)$. Here, we get the following cases:\\

    \textbf{Case I :} if $x\leq y$, then we get $N(a, x)\geq N(a, y)$. This implies that $N(a_1, x)-(a_1, x) \geq N(a_1, y)-(a_1, y)$, which contradicts our assumption. \\

    \textbf{Case II :} if $x\geq y$, then we get $N(a, x)\leq N(a, y)$. This implies that $N(a_1, x)-(a_1, x) \leq N(a_1, y)-(a_1, y)$, which is again a contradiction. Thus, $N$ has at most  one equilibrium point with respect to the parameter $a_1$.\\

    So, we get that  $N$ has at most one equilibrium point for each parameter $a_i$, $i=1, 2, \dots, n$. Hence, the fuzzy soft negation $N$ has at most $n$ equilibrium points.
\end{proof}

Thus, by following the first part of the above proof, we can easily prove the following theorem:\\
\begin{theorem}
    If a fuzzy soft negation operation $N:A\times[0, 1] \to A\times[0, 1]$ has an equilibrium point with respect to a parameter $a \in A$, then the equilibrium point must be unique.
\end{theorem}

\begin{definition}
    Let $T:(A\times[0, 1])\times (B\times[0, 1]) \to (A\times B)\times[0, 1]$ be a fuzzy soft t-norm. Then, $\forall a\in A$ and $\forall b\in B$, we have the following notions:\\
    
    (i) an element $x\in [0, 1]$ is called an idempotent element of $T$ if $T((a, x), (b, x))\\=((a, b), x)$ and the pair $((a, x), (b, x))$ is called the idempotent pair of the fuzzy soft t-norm $T$.\\
    
    (ii) an element $x\in [0, 1]$ is called a nilpotent element of $T$ if $T((a, x), (b, x))=((a, b), 0)$ and the pair $((a, x), (b, x))$ is called the nilpotent pair of the fuzzy soft t-norm $T$.\\
    
    (iii) an element $x\in (0, 1]$ is called a zero divisor of $T$ if there exists $y\in (0, 1]$ such that $T((a, x), (b, y))=((a, b), 0)$.\\
    
\end{definition}

\begin{theorem}
    Let $T$ be a fuzzy soft t-norm. Then, 0 and 1 are always idempotent elements of $T$.
\end{theorem}
\begin{proof}
    Let $T:(A\times[0, 1])\times (B\times[0, 1]) \to (A\times B)\times[0, 1]$ be a fuzzy soft t-norm. So, $T((a, 1), (b, x))=((a, b), x)$, $\forall a\in A, b\in B$ and $x\in [0, 1] $. Thus, for $x=1$, we get $T((a, 1), (b, 1))=((a, b), 1)$,  $\forall a\in A, b\in B$. This, implies that 1 is an idempotent element of $T$.\\

    Now, we assume that 0 is not an idempotent element of $T$. Then, there exists some $x\in (0, 1]$ such that $T((a, 0), (b, 0))=((a, b), x)$. Also, we have $T((a, 1), (b, 0))=((a, b), 0)$. So, $T((a, 0), (b, 0))\leq T((a, 1), (b, 0))\implies ((a, b), x)\leq ((a, b), 0)$, which is a contradiction. Therefore, we get $x=0$. Hence, $T((a, 0), (b, 0))=((a, b), 0)$, i.e., 0 is an idempotent element of $T$. 
\end{proof}

\begin{theorem}
     In a fuzzy soft t-norm, every non-zero nilpotent element is zero divisor of that fuzzy soft t-norm.  
\end{theorem}
\begin{proof}
    Let $x\in (0, 1]$ be a nilpotent element of a fuzzy soft t-norm $T:(A\times[0, 1])\times (B\times[0, 1]) \to (A\times B)\times[0, 1]$. Then, there exist $a\in A$ and $b\in B$ such that $T((a, x), (b, x))=((a, b), 0)$. It implies that $x$ is a zero divisor of $T$. Hence, every nilpotent element $x\in (0, 1]$ of a fuzzy soft t-norm $T$ is zero divisor of $T$.\\

It is important to note that converse of the above theorem is not true in general. Thus, we have the following example:\\

 \begin{example}
      Consider a fuzzy soft t-norm $T:(A\times[0, 1])\times (B\times[0, 1]) \to (A\times B)\times[0, 1]$ defined as $T((a, x), (b, y))=((a, b), max\{x+y-1, 0\})$. Now, if we take $x=0.8$ and $y=0.1$ then we get, $T((a, 0.8), (b, 0.1))=((a, b), 0)$. It implies that 0.8 and 0.1 are the zero divisors of the fuzzy soft t-norm $T$. But, $T((a, 0.1), (b, 0.1))=((a, b), 0)$ and $T((a, 0.8), (b, 0.8))=((a, b), 0.6)$. It suggests that 0.8 is not a nilpotent element of $T$. Therefore, it is not necessary for each zero divisor of $T$ to be a nilpotent element.
 \end{example}  
\end{proof}
\begin{theorem}
    Let $T, S:(A\times[0, 1])\times (B\times[0, 1]) \to (A\times B)\times[0, 1]$ be two operations such that, $T((a, x), (b, y))=((a, b), f(x, y))$ and $S((a, x), (b, y))=((a, b), g(x, y))$ where $f, g:[0, 1]\times [0, 1] \to [0, 1]$  are two functions such that $g(x, y)=1-f(1-x, 1-y)$. If $T$ is a fuzzy soft t-norm, then $S$ is a fuzzy soft t-conorm.
\end{theorem}
\begin{proof}
    We assume that $T$ is a fuzzy soft t-norm. So, $T$ satisfies all the conditions of definition 3.1. Now, we have to show that $S((a, x), (b, y)) = ((a, b), g(x, y))$ is fuzzy soft t-conorm. To show this, we just need to prove that $S$ satisfies the five conditions of definition 3.3.\\

    (i) Since $T$ is a  fuzzy soft t-norm, we get
    $T((a, 1), (b, y))=((a, b), y)\implies ((a, b), f(1, y))=((a, b), y)$. It implies that $f(1, y)=y$, $\forall y\in [0, 1]$.\\

    Now, $S((a, 0), (b, y))=((a, b), g(0, y))=((a, b), 1-f(1, 1-y))=((a, b), 1-\{1-y\})=((a, b), y)$.\\
    
    (ii) Again, $T((a, x), (b, 1))=((a, b), x)\implies ((a, b), f(x, 1))=((a, b), x)$. It implies that $f(x, 1)=x$, $\forall x\in [0, 1]$.\\

    Now, $S((a, x), (b, 0))=((a, b), g(x, 0))=((a, b), 1-f(1-x, 1))=((a, b), 1-\{1-x\})=((a, b), x)$.\\

    (iii)  $T((a, x), (b, y))=T((b, y), (a, x))\implies ((a, b), f(x, y))= ((b, a), f(y, x))$, $\forall x, y \in [0, 1]$. Hence, $f(x, y)=f(y, x), \forall x, y \in [0, 1]$. Also if $x, y \in [0, 1]$ then $1-x, 1-y \in [0, 1]$. Therefore, $f(1-x, 1-y)=f(1-y, 1-x)$.\\

    Now, $S((a, x), (b, y))=((a, b), g(x, y))=((a, b), 1-f(1-x, 1-y))$ and $S((b, y), (a, x))=((b, a), g(y, x))=((b, a), 1-f(1-y, 1-x)$. We already know that the sets of parameters $A\times B$ and $B\times A$ are indistinguishable. Also we have, $f(1-x, 1-y)=f(1-y, 1-x)$. Thus, $1-f(1-x, 1-y)= 1-f(1-y, 1-x)$. Finally, we get $((a, b), 1-f(1-x, 1-y)=((b, a), 1-f(1-y, 1-x)$, i.e., $S((a, x), (b, y))=S((b, y), (a, x))$.\\

    (iv) For all $x, y, z\in [0, 1]$, we get $T((a, x), T((b, y), (c, z)))=T(T((a, x), (b, y)), (c, z))$. It implies that $T((a, x), ((b, c), f(y, z))) = T(((a, b), \\f(x, y)), (c, z))$. This implies that $((a, (b, c)), f(x, f(y, z)))=(((a, b), c), f(f(x, y)\\, z))$. Hence, we get $f(x, f(y, z))=f(f(x, y), z)$, $\forall x, y, z \in [0, 1]$.\\

    Now, $S((a, x), S((b, y), (c, z)))=S((a, x), ((b, c), g(y, z)))=((a, (b, c)), g(x,\\ g(y, z)))$. Here, $g(x, g(y, z))=g(x, 1-f(1-y, 1-z))$. Now, we put $f(1-y, 1-z)=M$. So, we get $g(x, g(y, z))=g(x, 1-M)=1-f(1-x, 1-(1-M))=1-f(1-x, M)=1-f(1-x, f(1-y, 1-z))$.\\

    Again, $S(S((a, x), (b, y)), (c, z))=S(((a, b), g(x, y)), (c, z))=(((a, b), c), g(g\\(x, y), z))=(((a, b), c), g(1-f(1-x, 1-y), z))$. Now, we put $f(1-x, 1-y)=N$. So, we get $g(1-f(1-x, 1-y), z)=g(1-N, z)=1-f(1-(1-N), 1-z)=1-f(N, 1-z)=1-f(f(1-x, 1-y), 1-z)$.\\

    Since $f(x, f(y, z))=f(f(x, y), z)$, $\forall x, y, z \in [0, 1]$, thus $f(1-x, f(1-y, 1-z))=f(f(1-x, 1-y), 1-z)$. It implies  $1-f(1-x, f(1-y, 1-z))=1-f(f(1-x, 1-y), 1-z)$ and hence, $g(x, g(y, z))=g(g(x, y), z)$. Since the sets of parameters $A\times(B\times C)$ and $(A\times B)\times C$ are indistinguishable, we can write that $S((a, x), S((b, y), (c, z))) = S(S((a, x), (b, y)), (c, z))$.\\

    (v) We consider four real numbers $x_1, x_2, y_1, y_2 \in [0, 1]$ such that $x_1 \leq x_2$ and $y_1 \leq y_2$. Since $T$ is a fuzzy soft t-norm, we have $T((a, x_1), (b, y_1)) \leq T((a, x_2), (b, y_2))$. It implies that $((a, b), f(x_1, y_1))\leq ((a, b), f(x_2, y_2))$. Thus, $f(x_1, y_1))\leq f(x_2, y_2)$.\\

    Now, we have to show that $S((a, x_1), (b, y_1)) \leq S((a, x_2), (b, y_2))$. We have, $f(x_1, y_1))\leq f(x_2, y_2)$. It implies $  f(1-x_1, 1-y_1))\geq f(1-x_2, 1-y_2)$. So, $ 1-f(1-x_1, 1-y_1))\leq 1-f(1-x_2, 1-y_2)$. It gives $  g(x_1, y_1) \leq g(x_2, y_2)$. Hence, $((a, b), g(x_1, y_1))\leq ((a, b), g(x_2, y_2))$. Therefore, $S((a, x_1), (b, y_1)) \leq S((a, x_2), (b, y_2))$.     Hence, $S$ is a fuzzy soft t-conorm.
    
\end{proof}

Now, we provide the following definition of fuzzy soft implication: 

\begin{definition}
    An operation $K:(A\times[0, 1])\times (B\times[0, 1]) \to (A\times B)\times[0, 1]$ is said to be a fuzzy soft implication if it satisfies the following conditions:

    (i) if $x\leq z$, then $K((a, x), (b, y))\geq K((a, z), (b, y))$;

    (ii) if $y\leq z$, then $K((a, x), (b, y))\leq K((a, x), (b, z))$;

    (iii) $K((a, 1), (b, y))=((a, b), y)$, $\forall y \in [0, 1]$;

    (iv) $K((a, 0), (b, y))=((a, b), 1)$, $\forall y \in [0, 1]$;

    (v) $K((a, x), K((b, y), (c, z)))=K((b, y), K((a, x), (c, z)))$.\\

    Here, $a\in A, b\in B, c\in C$ and $x, y, z \in [0, 1]$.
\end{definition}

\begin{example}
    Let $K:(A\times[0, 1])\times (B\times[0, 1]) \to (A\times B)\times[0, 1]$ be an operation defined as $K((a, x), (b, y))=((a, b), min\{1, 1-x+y\})$. Then, $K$ is a fuzzy soft implication.
\end{example}

\begin{theorem}
    An operation $K:(A\times[0, 1])\times (B\times[0, 1]) \to (A\times B)\times[0, 1]$ is defined as $K((a, x), (b, y))= ((a, b), h(x, y))$. Then, the operation $K$ is said to be a fuzzy soft implication if the function $h:[0,1]\times [0, 1] \to [0, 1]$ is a fuzzy implication.
\end{theorem}
\begin{proof}
    Let $h:[0,1]\times [0, 1] \to [0, 1]$ be a fuzzy implication. So, it satisfies conditions of definition 2.6.\\

(I) Let $x\leq z$. Now, $K((a, x), (b, y)) = ((a, b), h(x, y)) \geq ((a, b), h(z, y)) = K((a, z), (b, y))$. So, $K((a, x), (b, y))\geq K((a, z), (b, y))$.\\

(II) Let $y\leq z$. Now, $K((a, x), (b, y)) = ((a, b), h(x, y) \leq ((a, b), h(x, z)) = K((a, x), (b, z))$. So, $K((a, x), (b, y))\leq K((a, x), (b, z))$.\\

(III) $K((a, 1), (b, y)) = ((a, b), h(1, y)) = ((a, b), y)$.\\

(IV) $K((a, 0), (b, y) ) = ((a, b), h(0, y)) = ((a, b), 1)$.\\

(V) We have $K((a, x), K((b, y), (c, z)))=K((a, x), ((b, c), h(y, z)))=((a, (b, c)), \\h(x, h(y, z)))$ and $K((b, y), K((a, x), (c, z)))=K((b, y), ((a, c), h(x, z)))=((b, (a, c)),\\ h(y, h(x, z)))$. We already know that the sets of parameters $A\times (B\times C)$ and $B\times (A\times C)$ are indistinguishable in soft theory and fuzzy soft theory, since $A\times B$ and $B\times A$ are indistinguishable. Thus, using the above condition (v), we get $K((a, x), K((b, y), (c, z)))=((b, (a, c)), h(y, h(x, z)))$.\\

Since $K$ satisfies all the conditions of definition 3.9, thus $K$ is a fuzzy soft implication.\\
    
\end{proof}

\section{Conclusion}

In this paper, we addressed the foundational incorrectness in Ali and Shabir's \cite{6} logical connectives of fuzzy soft set theory. The definitions introduced by Ali and Shabir \cite{6} deviated from the principles laid out in Molodtsov's original frameworks \cite{1,2,3,4,10,12}. Through this paper, we proposed revised definitions of t-norm, t-conorm, fuzzy soft negation, and fuzzy soft implication. We named them fuzzy soft t-norm, fuzzy soft t-conorm, fuzzy soft negation, and fuzzy soft negation. These definitions are consistent with the core philosophy of soft set theory and its hybrid structures with fuzzy set theory. Our contributions not only rectify inconsistencies in earlier approaches but also pave the way for a more systematic and reliable framework of logical connectives in fuzzy soft set theory. We provided suitable examples and proved several theorems. We also established a relation between fuzzy soft t-norm and fuzzy soft t-conorm. As previously mentioned, Molodtsov's cautionary advice regarding the hybrid structures of soft set theory and fuzzy set theory has been acknowledged. Therefore, in this paper, we have limited our focus to presenting accurate definitions, examples, and results, without delving into any applications.\\

\noindent
These advancements enhance the logical foundation of fuzzy soft set theory while expanding its potential applications in areas such as decision-making, optimization, control systems, data analysis, and more. Future research may focus on extending these concepts further, exploring their computational implementations, and investigating their applicability across diverse domains. Such efforts would help unlock the full potential of accurate logical connectives in fuzzy soft set theory as a powerful tool for managing uncertainty in complex systems.\\

{\bf Competing interests:} The authors declare that there is no competing interest.\\

{\bf Funding:} The authors did not receive any funding for this research work. 


\begin{thebibliography}{}

   \bibitem{1}
   Molodtsov, D. (1999). Soft set theory—first results. Computers $\&$ mathematics with applications, 37(4-5), 19-31.
   \bibitem{2}
   Molodtsov D. A. Structure of soft sets, Nechetkie Sistemy i Myagkie Vychisleniya, 2017, vol. 12, no. 1, pp. 5-18 (in Russian).
   \bibitem{3}
   Molodtsov, D. A. (2018). Equivalence and correct operations for soft sets. International Robotics $\&$ Automation Journal, 4(1), 18-21.
   \bibitem{4}
   Acharjee S., Oza A. Correct structures and similarity measures of soft sets along with historic comments of Prof. D.A. Molodtsov, Vestnik Udmurtskogo Universiteta. Matematika. Mekhanika. Komp'yuternye Nauki, 2023, vol. 33, issue 1, pp. 32-53.
   \bibitem{5}
   Zadeh, L. A. (1965). Fuzzy sets. Information and control, 8(3), 338-353.
   \bibitem{6}
   Ali, M. I.,  Shabir, M. (2013). Logic connectives for soft sets and fuzzy soft sets. IEEE Transactions on Fuzzy Systems, 22(6), 1431-1442.
   
   \bibitem{7}
   Maji, P. K., Biswas, R. K.,  Roy, A. (2001). Fuzzy soft sets.
   
   \bibitem{8}
   Roy, A. R., Maji, P. K. (2007). A fuzzy soft set theoretic approach to decision making problems. Journal of computational and Applied Mathematics, 203(2), 412-418.
  
  \bibitem{9}
   Acharjee, S.,  Molodtsov, D. A. (2021). Soft rational line integral, Vestnik Udmurtskogo Universiteta. Matematika. Mekhanika. Komp'yuternye Nauki, 2021,  vol.31, no. 4, pp. 578-596.

   \bibitem{10}
    Acharjee, S.,  Medhi, S. (2024). The correct structures in fuzzy soft set theory. arXiv preprint arXiv:2407.06203.

\bibitem{11}
Maji, P. K., Biswas, R.,  Roy, A. R. (2003). Soft set theory. Computers $\&$ mathematics with applications, 45(4-5), 555-562.

\bibitem{12}
Molodtsov, D. A., Soft set theory, URSS publishing group, 2004, Russia. (In Russian)

\bibitem{13}
Balasubramaniam, J., $\&$ Rao, C. J. M. (2004). On the distributivity of implication operators over T and S norms. IEEE Transactions on Fuzzy Systems, 12(2), 194-198.


    
\end{thebibliography}
\end{document}